\documentclass{svjour3}                     
\smartqed  
\usepackage{graphicx}
\usepackage{amsfonts,amsmath,latexsym,amssymb}

\newcommand{\m}{m}

\newcommand{\IZ}{\mathbb Z}

\newcommand{\IQ}{\mathbb Q}

\newcommand{\ord}{\mathop{\rm ord}\nolimits}

\newcommand{\Oseen}{{\mathcal{O}}}
\newcommand{\A}{{\mathfrak{A}}}

\newcommand{\D}{{\mathfrak{D}}}
\newcommand{\B}{{\mathfrak{B}}}

\newcommand{\Qbar}{\overline{\IQ}}

\newcommand{\Diff}{\D}

\begin{document}

\title{On certain infinite extensions of the rationals with Northcott property \thanks{This work has been financially supported by the Swiss National Science Foundation.}
}

\titlerunning{On infinite extensions with Northcott property}        

\author{Martin Widmer}


\institute{M. Widmer \at
             Department of Mathematics, University of Texas at Austin,\\
	     1 University Station C1200, Austin, TX 78712, USA\\
              \email{Widmer@math.utexas.edu}\\           
             \emph{Present address:}\\
              Institut f\"ur Mathematik A, Technische Universit\"at Graz, \\
	      Steyrergasse 30/II, 8010 Graz, Austria \\
              Tel.: +43-316-8737620\\
              Fax: +43-316-8737126 \\
              \email{Widmer@tugraz.at}     
           }

\date{Received: date / Accepted: date}

\maketitle

\begin{abstract}
A set of algebraic numbers has the Northcott property if each of its subsets of bounded Weil height is finite.
Northcott's Theorem, which has many Diophantine applications, states that sets of bounded degree have the Northcott property.
Bombieri, Dvornicich and Zannier raised the problem of finding fields of infinite degree with this property.
Bombieri and Zannier have shown that $\IQ_{ab}^{(d)}$, the maximal abelian subfield of the field generated by all algebraic numbers of degree at most $d$,
is such a field. In this note we give a simple criterion for the Northcott property and, as an application, we deduce several new examples, e.g. $\IQ(2^{1/d_1},3^{1/d_2},5^{1/d_3},7^{1/d_4},11^{1/d_5},...)$ has the Northcott property if and only if $2^{1/d_1},3^{1/d_2},5^{1/d_3},7^{1/d_4},11^{1/d_5},...$ tends to infinity.
\keywords{Northcott property \and height \and preperiodic points \and field arithmetic}
 \subclass{12F05 \and 11G50 }
\end{abstract}

\section{Introduction}
\label{intro}
Let $\mathcal{A}$ be a subset of the algebraic numbers $\Qbar$ and denote by $H(\cdot)$ the non-logarithmic
absolute Weil height on $\Qbar$ as defined in \cite{BG}.
Following Bombieri and Zannier \cite{BoZa} we say $\mathcal{A}$
has the Northcott property, short property $(N)$,
if for each positive real number $X$ there are only finitely many elements $\alpha$ in $\mathcal{A}$
with $H(\alpha)\leq X$. The $1$-dimensional version of Northcott's Theorem (see \cite{26} Theorem 1)
states that sets of algebraic numbers with uniformly bounded degree (over $\IQ$) have property $(N)$.
Northcott's Theorem has been used extensively, especially to deduce finiteness results in Diophantine geometry.
Other applications will be mentioned briefly in Section \ref{sec:0}.
Bombieri and Zannier \cite{BoZa} and more explicitly Dvornicich and Zannier (\cite{DvZa} p.165) proposed the problem of 
finding other fields than number fields with property $(N)$. 
In this note we give a simple sufficient criterion for an infinite extension of $\IQ$
to have property $(N)$.
Our criterion depends on the growth rate of certain discriminants. The method  
uses a lower bound due to Silverman for the height of an element generating the number field.
As an application we deduce property $(N)$ for several infinite extensions, here is just one example; with positive integers $d_i$ the extension $\IQ(2^{1/d_1},3^{1/d_2},5^{1/d_3},7^{1/d_4},11^{1/d_5},...)$ has property $(N)$ if and only if $2^{1/d_1},3^{1/d_2},5^{1/d_3},7^{1/d_4},11^{1/d_5},...$ tends to infinity.\\

For an arbitrary number field $K$ and a positive integer $d$
let $K^{(d)}$ be the compositum of all field extensions of $K$ of degree at most $d$.
Bombieri and Zannier \cite{BoZa} addressed the following question: \it does $K^{(d)}$ have property $(N)$? \rm 
So far the only contribution to this question is due to Bombieri and Zannier (\cite{BoZa} Theorem 1.1).
Let us write $K^{(d)}_{ab}$ for the compositum of all abelian extensions $F/K$ with $K\subseteq F\subseteq K^{(d)}$.
\begin{theorem}[Bombieri, Zannier]\label{ThBoZa}
The field $K^{(d)}_{ab}$ has property $(N)$, for any positive integer $d$.
\end{theorem}
Since $K^{(2)}=K^{(2)}_{ab}$ Theorem \ref{ThBoZa} positively answers Bombieri and Zannier's question for $d=2$.
However, for $d>2$ the question whether $K^{(d)}$ has property $(N)$ remains open.
Another consequence of Theorem \ref{ThBoZa} is the following result.
\begin{corollary}[Bombieri, Zannier]\label{cor2BoZa}
For any positive integer $d$ the field 
\newline$\IQ(1^{1/d},2^{1/d},3^{1/d},4^{1/d},5^{1/d},...)$ has property $(N)$.
\end{corollary}
Dvornicich and Zannier (\cite{DvZa} Theorem 2.1) observed that by a small variation of Northcott's argument the ground field $\IQ$ in Northcott's Theorem can be replaced by any field with property $(N)$.
This turns out to be a very useful fact so that we state it explicitly as a theorem.
\begin{theorem}\label{ThDvZa}
Let $L$ be a field of algebraic numbers with property $(N)$ and let $d>0$ be an integer. The set of algebraic numbers of 
degree at most $d$ over $L$ has property $(N)$. In particular every finite extension of $L$ has the property $(N)$.
\end{theorem}
Taking a finite extension of a field with property $(N)$
is of course a very special case of taking the compositum of two fields with property $(N)$. So one might ask: is the property $(N)$
preserved under taking the compositum of two fields? We shall see that this is not always the case.\\

Before we state our own results let us fix some basic notation.
All fields are considered to lie in a fixed algebraic closure of $\IQ$.
For positive rational integers $a,b$ the expression $a^{1/b}$ denotes the real positive $b$-th root of $a$, unless stated otherwise.
By a prime ideal we always mean a non-zero prime ideal.
Let $F,M,K$ be number fields with $F\subseteq M \subseteq K$ and write $\Oseen_K$ for the ring of integers in $K$.
For a non-zero fractional ideal $\A$ of $\Oseen_K$ in $K$ let $D_{K/M}(\A)$
be the discriminant-ideal of $\A$ relative to $M$ (for the definition see \cite{13} p.65)	
and write $D_{K/M}$ for $D_{K/M}(\Oseen_K)$ (see also \cite{Neu} p.201).
Let us denote by $N_{M/F}(\cdot)$ the norm from $M$ to $F$ as defined in \cite{13} p.24. Then we have
\begin{alignat}1
\label{refdis}
D_{K/F}=D_{M/F}^{[K:M]}N_{M/F}(D_{K/M})
\end{alignat}
(see \cite{Neu} (2.10) Corollary p.202).
For a non-zero fractional ideal $\B$ of $\Oseen_M$ in $M$ we interpret the principal ideal $N_{M/\IQ}(\B)$ 
as the unique positive rational generator of this ideal.
Note that $D_{K/M}$ is an integral ideal
in $\Oseen_{M}$ and thus 
$N_{M/\IQ}(D_{K/M})$ is in $\IZ$.
We write $\Delta_K$ for the absolute discriminant of $K$ so that $D_{K/\IQ}$ is the principal ideal generated by 
$\Delta_K$. In particular (\ref{refdis}) yields
\begin{alignat}1
\label{dis}
|\Delta_{K}|=|\Delta_{M}|^{[K:M]}N_{M/\IQ}(D_{K/M}).
\end{alignat}
We will also frequently use the following fact (see \cite{Hilbert} Theorem 85 p.97): let $F,K$ be two number fields.
A prime $p$ in $\IZ$ ramifies in the compositum of $F$ and $K$ if and only if
it ramifies in $F$ or in $K$.\\

So far Theorem \ref{ThBoZa}, and its immediate consequences, were the only sources for fields of infinite degree with property $(N)$.
Our first result is a simple but rather general criterion for the property $(N)$ concerning subfields of $\Qbar$. Roughly speaking it states that the union of fields in a saturated (i.e. without intermediate fields) nested sequence of number fields with enough ramification at each step has property $(N)$. 
\begin{theorem}\label{Th1}
Let $K$ be a number field, let $K=K_0\subsetneq K_1\subsetneq K_2\subsetneq....$ be a nested sequence of finite extensions and set $L=\bigcup_{i}K_i$. Suppose that
\begin{alignat}1\label{reldiscond0}
\inf_{K_{i-1}\subsetneq M \subseteq K_i}\left(N_{K_{i-1}/\IQ}(D_{M/K_{i-1}})\right)^{\frac{1}{[M:K_0][M:K_{i-1}]}}\longrightarrow \infty
\end{alignat}
as $i$ tends to infinity where the infimum is taken over all intermediate fields $M$ strictly larger than $K_{i-1}$.
Then  the field $L$ has the Northcott property.
\end{theorem}
If the nested sequence of number fields is saturated then (\ref{reldiscond0}) simplifies to
\begin{alignat}1\label{reldiscond}
N_{K_{i-1}/\IQ}(D_{K_i/K_{i-1}})^{\frac{1}{[K_{i}:K_0][K_i:K_{i-1}]}}\longrightarrow \infty.
\end{alignat}
In the sequel we give several applications of Theorem \ref{Th1}.
For a number field $K$ and a prime ideal $\wp$ of $\Oseen_K$ we say $D=x^d+a_{1}x^{d-1}+...+a_{d}$ in $\Oseen_K[x]$
is a $\wp$-Eisenstein polynomial if $a_j\in \wp$ for $1\leq j\leq d$ and $a_d\notin \wp^2$. Such a polynomial is irreducible over $K$
(see \cite{Marcus} p.256).
As a consequence of Theorem \ref{Th1} we deduce the following theorem.
\begin{theorem}\label{corEisenstein}
Let $K$ be a number field,
let $p_1,p_2,p_3,...$ be a sequence of positive prime numbers
and for $i=1,2,3,...$ let $D_i$ be a $p_i$-Eisenstein polynomial in $\IZ[x]$. 
Denote $\deg D_i=d_i$ and let $\alpha_i$ be any root of $D_i$.
Moreover suppose that $p_i\nmid \Delta_{\IQ(\alpha_j)}$ for $1\leq j<i$ and that $p_i^{{1}/{d_i}}\longrightarrow \infty$
as $i$ tends to infinity.
Then  the field $K(\alpha_1,\alpha_2,\alpha_3,...)$ has the Northcott property.
\end{theorem}
Theorem \ref{corEisenstein} implies a refinement of Corollary \ref{cor2BoZa}. This refinement
shows that the condition $p_i^{{1}/{d_i}}\longrightarrow \infty$ in Theorem \ref{corEisenstein} cannot be weakened.
\begin{corollary}\label{cor2}
Let $K$ be a number field,
let $p_1<p_2<p_3<...$ be a sequence of positive primes
and let $d_1,d_2,d_3,...$ be a sequence of positive integers.
Then  the field $K(p_1^{1/d_1},p_2^{1/d_2},p_3^{1/d_3},...)$ has the Northcott property if and only if
$|p_i^{{1}/{d_i}}|\longrightarrow \infty$ as $i$ tends to infinity. Here $p_i^{{1}/{d_i}}$ is any $d_i$-th root of $p_i$ and
$|\cdot|$ denotes the complex modulus.
\end{corollary}
If the $d_i$ are prime and not uniformly bounded then $\IQ(p_1^{1/d_1},p_2^{1/d_2},p_3^{1/d_3},...)$ contains elements of arbitrarily large prime degree and thus it cannot be generated over $\IQ$ by algebraic numbers of bounded degree. The conclusion remains true if we drop the primality condition on $d_i$. This can be deduced from Proposition 1 in \cite{BoZa} which implies for any 
subfield  $L\subseteq \IQ^{(d)}$ the local degrees $[L_v:\IQ_v]$ are bounded solely in terms of $d$. 
Now the local degrees of $L=\IQ(p_1^{1/d_1},p_2^{1/d_2},p_3^{1/d_3},...)$ are not uniformly bounded and so $L$ is not contained in $\IQ^{(d)}$ for any choice of $d$.
To the best of the author's knowledge Corollary \ref{cor2} provides the first such example of a field with property $(N)$.
Moreover, Corollary \ref{cor2} easily implies the following statement.
\begin{theorem}\label{thmcompfields}
Property $(N)$ is not generally preserved under taking the composite of two fields. More concretely:
let $p_i$ be the $i+1$-th prime number and 
set $d_i=[\sqrt{\log p_i}]$. Let
\begin{alignat*}1
L_1&=\IQ(p_1^{1/d_1},p_2^{1/d_2},p_3^{1/d_3},...),\\
L_2&=\IQ(p_1^{1/(d_1+1)},p_2^{1/(d_2+1)},p_3^{1/(d_3+1)},...).
\end{alignat*}
Then  $L_1$ and $L_2$ both have property $(N)$ but their composite field does not have property $(N)$.
\end{theorem}
Another example proving Theorem \ref{thmcompfields},  again coming from Corollary \ref{cor2}, is
as follows: consider the fields $L_1=\IQ(p_1^{1/d_1},p_2^{1/d_2},p_3^{1/d_3},...)$ and
$L_2=\IQ(\zeta_1p_1^{1/d_1},\zeta_2p_2^{1/d_2},\zeta_3p_3^{1/d_3},...)$, where $d_i$ is as in Theorem \ref{thmcompfields} and
$\zeta_i$ are primitive $d_i$-th roots of unity. Then plainly $L_1,L_2$ have the
property $(N)$ (by Corollary \ref{cor2}) but $L_1L_2$ does not because it contains infinitely many
roots of unity.\\

Let us give one more immediate consequence of Theorem \ref{corEisenstein}. This result can be considered as a
very small step towards the validity of property $(N)$ for $K^{(d)}$.
\begin{corollary}\label{corEisenstein3}
Let $d$ be a positive integer, let $F_0$ be an arbitrary number field
and let $F_1,F_2,F_3,...$ be a sequence of finite extensions of $F_0$ with $[F_i:F_0]\leq d$.
Moreover suppose there exists a sequence $p_1,p_2,p_3,...$ of positive prime numbers
such that $p_i$ ramifies totally in $F_i$ but does not
ramify in $F_j$ for $1\leq j<i$. Then the compositum of $F_0,F_1,F_2,F_3,...$ has the Northcott property.
\end{corollary}
In the case $d=3$ one can apply the criterion from Theorem \ref{Th1} directly to prove a stronger result.
\begin{corollary}\label{corcorTh1}
Let $F_0$ be an arbitrary number field and let $F_1,F_2,F_3,...$ be a sequence of field extensions of $F_0$ with $[F_i:F_0]\leq 3$ such that for each positive integer $i$ there is a
prime $p_i$ with $p_i\mid \Delta_{F_i}$ and $p_i\nmid \Delta_{F_j}$ for $0\leq j<i$. 
Then the compositum of
$F_0,F_1,F_2,F_3,...$ has the Northcott property.
\end{corollary}
As a next step towards $K^{(3)}$ we would like to replace $F_i$ in Corollary \ref{corcorTh1} by its Galois closure $F_i^{(g)}$ over $F_0$.
Unfortunately we have to impose an additional technical condition and we also restrict $F_0$ to $\IQ$.
\begin{corollary}\label{cor2corTh1}
Let $F_1,F_2,F_3,...$ be a sequence of field extensions with $[F_i:\IQ]\leq 3$ such that for each integer $i>1$ there is a
prime $p_i$ with $p_i\mid \Delta_{F_i}$ and $p_i\nmid \Delta_{F_j}$ for $1\leq j<i$. 
Furthermore suppose that for each $i>1$ at least one of the following conditions
holds:\\
(a)$F_i/\IQ$ is Galois.\\
(b)$F_i=\IQ(\alpha)$ for an $\alpha$ with $\alpha^3 \in \IQ$.\\
(c)$F_i=\IQ(\alpha)$ for an algebraic integer $\alpha$ with $2\nmid \ord_{p_i}Disc(D_{\alpha})$ for the monic minimal polynomial $D_{\alpha}\in \IZ[x]$ of $\alpha$.\\
(d)$F_i=\IQ(\alpha)$ for a root $\alpha$ of a polynomial of the form $x^3+a_0c^3x^l+b_0^lc^3$ with $l\in \{1,2\}$ and rational integers $a_0,b_0,c$ satisfying $\gcd(2a_0c,3b_0)=1$.\\
Then the compositum of $F_1^{(g)},F_2^{(g)},F_3^{(g)},...$ has the Northcott property.
\end{corollary}
Theorem \ref{corEisenstein}, Corollary \ref{corEisenstein3}, Corollary \ref{corcorTh1}, and Corollary \ref{cor2corTh1}  can be generalized in various ways, for instance the constraints in these results can be relaxed by computing the contribution to the relative discriminant of more than just one prime.

\section{A simple observation}
\label{sec:1}

Let $L$ be a field of algebraic numbers of infinite degree. Now we consider a nested sequence of fields
\begin{alignat*}1
K_0\subsetneq K_1\subsetneq K_2\subsetneq K_3\subsetneq...
\end{alignat*}
such that 
\begin{alignat*}3
&(i) &&\text{$K_0$ has the property $(N)$, }\\
&(ii) &&[K_i:K_{i-1}]<\infty \text{ for $i>0$, }\\
&(iii) &&L=\bigcup_{i=0}^{\infty}K_i.
\end{alignat*}
For a finite extension $M/F$ of subfields of $\Qbar$ we define
\begin{alignat*}1
\delta(M/F) = \inf\{H(\alpha);F(\alpha)=M\}.
\end{alignat*}
Note that if $M$ has the property $(N)$ then the infimum is attained, i.e. there exists
$\alpha\in M$ with $F(\alpha)=M$ and $H(\alpha)=\delta(M/F)$.\\ 
Since each $K_i$ is a finite extension
of $K_0$ we deduce by $(i)$ and Theorem \ref{ThDvZa} that each field $K_i$ has property $(N)$.
\begin{proposition}\label{propobserv}
$L$ has property $(N)$ if and only if 
\begin{alignat*}1
\inf_{K_{i-1}\subsetneq M \subseteq K_i}\delta(M/K_{i-1})\longrightarrow \infty \qquad \text{  as } i \rightarrow \infty
\end{alignat*}
where the infimum is taken over all intermediate fields $M$ strictly larger than $K_{i-1}$.
\end{proposition}
Although it is not needed here, we point out that for $i>0$
\begin{alignat*}1
\inf_{K_{i-1}\subsetneq M \subseteq K_i}\delta(M/K_{i-1})=\inf_{\alpha\in K_i\setminus K_{i-1}}H(\alpha)
\end{alignat*}
and this holds even if $K_i$ does not have property $(N)$.
The inequality ``$\leq$'' is obvious. For ``$\geq$'' let $M$ be a field with $K_{i-1}\subsetneq M \subseteq K_i$ and let $\alpha_1, \alpha_2, \alpha_3,...$ be a sequence in $M$ with 
$K_{i-1}(\alpha_j)=M$ and $H(\alpha_j)\rightarrow \delta(M/K_{i-1})$ as $j\rightarrow \infty$. Then clearly $\alpha_j\in K_i\backslash K_{i-1}$ and thus 
$H(\alpha_j)\geq \inf_{\alpha\in K_i\setminus K_{i-1}}H(\alpha)$. This shows that $\delta(M/K_{i-1})\geq \inf_{\alpha\in K_i\setminus K_{i-1}}H(\alpha)$
which proves the inequality ``$\geq$''.
\begin{proof}[of Proposition \ref{propobserv}]
For brevity let us write 
\begin{alignat*}1
A_i=\inf_{K_{i-1}\subsetneq M \subseteq K_i}\delta(M/K_{i-1}).
\end{alignat*}
First we show that property $(N)$ for the field $L$ implies $A_i \rightarrow \infty$.\\
For each $i>0$ we can find $\alpha_i\in K_i\backslash K_{i-1}$
with $H(\alpha_i)=A_i$, in particular the elements $\alpha_i$ are pairwise distinct.
Now suppose $(A_i)_{i=1}^{\infty}$ has a bounded subsequence.
Hence we get infinitely many elements $\alpha_i\in L$ with uniformly bounded height and so $L$ does not have property $(N)$.\\

Next we prove that $A_i \rightarrow \infty$ implies property $(N)$ for the field $L$.\\
Suppose $L$ does not have property (N). Hence there exists an infinite sequence
$\alpha_1,\alpha_2,\alpha_3,...$ of pairwise distinct elements in $L\backslash K_0$
with $H(\alpha_j)\leq X$ for a certain fixed real number $X$. Let $i=i(\alpha_j)$ be such that $\alpha_j \in K_i\backslash K_{i-1}$.  Thus
\begin{alignat*}1
K_{i-1}\subsetneq K_{i-1}(\alpha_j) \subseteq K_i
\end{alignat*}
and hence 
\begin{alignat*}1
A_i\leq \delta(K_{i-1}(\alpha_j)/K_{i-1})\leq H(\alpha_j)\leq X.
\end{alignat*}
Since each field $K_i$ has the property $(N)$ we conclude $i(\alpha_j)\longrightarrow \infty$ as $j \rightarrow \infty$. Thus $(A_i)_{i=1}^{\infty}$ has a bounded subsequence.
\end{proof}

\section{Silverman's inequality}
\label{sec:2}
In order to apply Proposition \ref{propobserv} we need a lower bound for the invariant $\delta(M/K)$.
A good lower bound was proven by Silverman if both fields are number fields.
So let $K,M$ be number fields with $K\subseteq M$ and $\m=[M:K]>1$.
Let $\alpha$ be a primitive point of the extension $M/K$, i.e. 
$M=K(\alpha)$. We apply Silverman's Theorem 2 from \cite{9} with $F=K$ and $K=M$ and
with Silverman's $S_F$ as the set of archimedean absolute values.
Then Silverman's $L_F(\cdot)$ is simply the usual norm
$N_{F/\IQ}(\cdot)$ and we deduce
\begin{alignat}1
\label{sil}
H(\alpha)^{[K:\IQ]}\geq \exp\left(-\frac{\delta_K\log \m}{2(\m-1)}\right)
N_{K/\IQ}(D_{M/K})^{\frac{1}{2\m(\m-1)}}
\end{alignat}
where $\delta_K$ is the number of archimedean places
of $K$. Since Silverman used an ``absolute height relative to $K$'' rather than an absolute height, we had 
to take the $[K:\IQ]$-th power on the left hand side of (\ref{sil}).

\section{Proof of Theorem \ref{Th1}}
\label{sec:3}
From Proposition \ref{propobserv} we know it suffices to show 
\begin{alignat*}1
\inf_{K_{i-1}\subsetneq M \subseteq K_i}\delta(M/K_{i-1})\longrightarrow \infty \qquad \text{  as } i \rightarrow \infty.
\end{alignat*}
So let $M$ be an intermediate field $K_{i-1}\subsetneq M \subseteq K_i$ and set $m=[M:K_{i-1}]$. We apply (\ref{sil}) with $K$ replaced by $K_{i-1}$.
Then taking the $[K_{i-1}:\IQ]$-th root and using $\delta_{K_{i-1}}\leq [K_{i-1}:\IQ]$ we conclude for any $\alpha \in M$ with $K_{i-1}(\alpha)=M$
\begin{alignat*}1
H(\alpha)\geq m^{-\frac{1}{2(m-1)}}(N_{K_{i-1}/\IQ}(D_{M/K_{i-1}}))
^{\frac{1}{2[K_{i-1}:\IQ]m(m-1)}}.
\end{alignat*}
In particular
\begin{alignat}1\label{delta}
\inf_{K_{i-1}\subsetneq M \subseteq K_i}\delta(M/K_{i-1})\geq (1/2)\inf_{K_{i-1}\subsetneq M \subseteq K_i}(N_{K_{i-1}/\IQ}(D_{M/K_{i-1}}))
^{\frac{1}{2[K_{i-1}:\IQ]m(m-1)}}.
\end{alignat}
Now using $[K_{i-1}:\IQ]m=[K_0:\IQ][M:K_0]$ and the hypothesis of the theorem 
we see that the right hand-side of (\ref{delta}) tends to
infinity as $i$ tends to infinity. This completes the proof of Theorem \ref{Th1}.

\section{Proof of Theorem \ref{corEisenstein}}
\label{sec:4}
Let us recall the following well-known lemma.
\begin{lemma}\label{Eisensteintotram}
Let $F,K$ be number fields with $F\subseteq K$. Let $\wp$ be a prime ideal in $\Oseen_F$.
The following are equivalent:\\
(i) $\wp$ ramifies totally in $K$.\\
(ii) $K=F(\alpha)$ for a root $\alpha$ of a $\wp$-Eisenstein polynomial in $\Oseen_F[x]$.
\end{lemma}
\begin{proof}
See for instance Theorem 24. (a) p.133 in \cite{FroTay}
\end{proof}
We can now prove Theorem \ref{corEisenstein}.
Let $K_0=K$ and for $i>0$ let $K_i=K_{i-1}(\alpha_i)$. 
By assumption we have $p_i\nmid \Delta_{\IQ(\alpha_j)}$ for $1\leq j<i$. 
Since $K_{i-1}$ is the compositum of $K_0,\IQ(\alpha_1),...,\IQ(\alpha_{i-1})$
we conclude that only primes dividing $\Delta_{K_0}\Delta_{\IQ(\alpha_1)}...\Delta_{\IQ(\alpha_{i-1})}$ 
can ramify in $K_{i-1}$. By assumption we know that $p_i\longrightarrow \infty$
which implies that there is an $i_0$ such that $p_i\nmid \Delta_{K_{i-1}}$ for all $i\geq i_0$.
Now we shift the index $i$ by $i_0$ steps so that the new $K_i,p_i,d_i$ are the old
$K_{i+i_0},p_{i+i_0},d_{i+i_0}$ and therefore $p_i$ is unramified in $K_{i-1}$ for all $i\geq 1$.
Now clearly $K_0\subsetneq K_1\subsetneq K_2\subsetneq....$ and of course $\bigcup_{i=0}^{\infty}K_i=K(\alpha_1,\alpha_2,\alpha_3,...)$. 
We will apply Theorem \ref{Th1} but first we have to make sure that
condition (\ref{reldiscond0}) holds.\\

Now let $i>0$ and let $M$ be an intermediate field with $K_{i-1}\subsetneq M\subseteq K_i$. Moreover set
$m=[M:K_{i-1}]$. Let $\wp$ be any prime ideal in $\Oseen_{K_{i-1}}$ above $p_i$.
Since $p_i$ is unramified in $K_{i-1}$ we conclude that 
$D_i$ is a $\wp$-Eisenstein polynomial in $\Oseen_{K_{i-1}}[x]$. 
According to the Eisenstein criterion 
this implies that $D_i$ is irreducible over $K_{i-1}$ and since $K_i=K_{i-1}(\alpha_i)$ we get $[K_i:K_{i-1}]=d_i$.
Moreover we conclude by Lemma \ref{Eisensteintotram} that $\wp$ ramifies
totally in $K_i/K_{i-1}$. Let 
\begin{alignat*}1
(p_i)=\wp_1...\wp_s
\end{alignat*}
be the decomposition into prime ideals in $\Oseen_{K_{i-1}}$.
Since $\wp_j$ ramifies totally in $K_i/K_{i-1}$ it also ramifies totally
in $M/K_{i-1}$. Hence 
\begin{alignat*}1
\wp_j=\B_j^{m} 
\end{alignat*}
for $1\leq j\leq s$ and prime ideals
$\B_j$ in $\Oseen_M$. Let $\Diff_{M/K_{i-1}}$ be the different of $M/K_{i-1}$ (for the definition see \cite{Neu} p.195).
Then we have 
$\B_j^{m-1}\mid \Diff_{M/K_{i-1}}$ (see \cite{Neu} (2.6) Theorem p.199)  and therefore
\begin{alignat*}1
(\B_1...\B_s)^{m-1}\mid \Diff_{M/K_{i-1}}.
\end{alignat*}
The discriminant $D_{M/K_{i-1}}$ is the norm of the different $\Diff_{M/K_{i-1}}$ from $M$ to $K_{i-1}$ (see \cite{Neu} 
(2.9) Theorem p.201). Taking then norms from $K_{i-1}$ to $\IQ$ we conclude
\begin{alignat*}1
N_{K_{i-1}/\IQ}(D_{M/K_{i-1}})=N_{K_{i-1}/\IQ}(N_{M/K_{i-1}}(\Diff_{M/K_{i-1}}))=N_{M/\IQ}(\Diff_{M/K_{i-1}}).
\end{alignat*}
Therefore
\begin{alignat}1\label{discdiv}
N_{M/\IQ}((\B_1...\B_s)^{m-1}) \mid N_{K_{i-1}/\IQ}(D_{M/K_{i-1}}).
\end{alignat}
On the other hand we have
\begin{alignat*}3
N_{M/\IQ}((\B_1...\B_s)^{m-1})&=(N_{M/\IQ}(\prod_{j=1}^{s}\B_j^m))^\frac{m-1}{m}&&=(N_{M/\IQ}(\prod_{j=1}^{s}\wp_j))^\frac{m-1}{m}\\
&=(N_{M/\IQ}(p_i))^\frac{m-1}{m}&&=p_i^{[K_{i-1}:\IQ](m-1)}.
\end{alignat*}
Combining the latter with (\ref{discdiv}) and not forgetting that $1<m=[M:K_{i-1}]\leq d_i$ we end up with 
\begin{alignat*}1
N_{K_{i-1}/\IQ}(D_{M/K_{i-1}})^{\frac{1}{[M:K_0][M:K_{i-1}]}}
\geq p_i^{\frac{[K_{i-1}:\IQ](m-1)}{[M:K_0]m}}
=p_i^{\frac{[K_{0}:\IQ](m-1)}{m^2}}\geq p_i^{\frac{1}{2m}}\geq p_i^{\frac{1}{2d_i}}.
\end{alignat*}
By hypothesis of the theorem $p_i^{\frac{1}{d_i}}$ tends to infinity.
Hence we can apply Theorem \ref{Th1} and this completes the proof of Theorem \ref{corEisenstein}.

\section{Proof of Corollary \ref{cor2}}
Since $H(p_i^{1/d_i})=|p_i^{1/d_i}|$ we see that condition $|p_i^{{1}/{d_i}}|\longrightarrow \infty$ is necessary to 
obtain property $(N)$.
Now let us prove that this condition implies property $(N)$.
The hypothesis implies that there is 
an $i_1$ such that $p_j>d_{j}$ for all $j>i_1$. Therefore we have
$p_i\geq p_j>d_{j}$ for all $i\geq j>i_1$. Clearly there exists an $i_2$ such that
$p_i>\max\{d_1,...,d_{i_1}\}$ for all $i\geq i_2$. Thus $p_i>\max\{d_1,...,d_{i_1},d_{i_1+1},...,d_{i}\}$ for all $i\geq i_0:=\max\{i_1,i_2\}$.
This implies
$p_i\nmid d_1p_1...d_{i-1}p_{i-1}$ for all $i\geq i_0$.
Set $D_i=x^{d_i}-p_i$, $\alpha_i=p_i^{1/d_i}$, $K_0=K$ and $K_i=K(\alpha_1,...,\alpha_i)$.
Since $\Delta_{\IQ(\alpha_{j})}$ divides $|Disc(D_j)|=d_j^{d_j}p_j^{d_j-1}$ 
we conclude $p_i\nmid \Delta_{\IQ(\alpha_{j})}$ for all $i\geq i_0$ and $1\leq j<i$. Now shift the index by $i_0$ steps,
more precisely: define $\widetilde{K}_{i}=K_{i_0+i}$, $\widetilde{p}_{i}=p_{i_0+i}$, $\widetilde{D}_{i}=D_{i_0+i}$, $\widetilde{d}_{i}=d_{i_0+i}$, $\widetilde{\alpha}_{i}=\alpha_{i_0+i}$.
Hence $\widetilde{p}_{i}\nmid \Delta_{\IQ(\widetilde{\alpha}_{j})}$ for all $i$ and $1\leq j<i$.
Clearly $K(\alpha_1,\alpha_2,\alpha_3,...)=\widetilde{K}_{0}(\widetilde{\alpha}_{1},\widetilde{\alpha}_{2},\widetilde{\alpha}_{3},...)$ and $|\widetilde{p}_{i}^{1/\widetilde{d}_{i}}|\longrightarrow \infty$.
Applying Theorem \ref{corEisenstein} with $K=\widetilde{K}_{0}$ and
$\widetilde{p}_{i},\widetilde{D}_{i},\widetilde{d}_{i},\widetilde{\alpha}_{i}$
completes the proof.

\section{Proof of Theorem \ref{thmcompfields}}
Note that $p_i^{1/d_i}$ and $p_i^{1/(d_i+1)}$ tend to infinity whereas $p_i^{1/(d_i^2+d_i)}$ is bounded
as $i$ tends to infinity.
Hence Corollary \ref{cor2} tells us that 
\begin{alignat*}1
L_1=\IQ(p_1^{1/d_1},p_2^{1/d_2},p_3^{1/d_3},...),\quad
L_2=\IQ(p_1^{1/(d_1+1)},p_2^{1/(d_2+1)},p_3^{1/(d_3+1)},...)
\end{alignat*}
both have property $(N)$. But $p_i^{1/d_i}/p_i^{1/(d_i+1)}=p_i^{1/(d_i^2+d_i)}$ and so the compositum of $L_1$ and $L_2$ contains
the field 
\begin{alignat*}1
\IQ(p_1^{1/(d_1^2+d_1)},p_2^{1/(d_2^2+d_2)},p_3^{1/(d_3^2+d_3)},...) 
\end{alignat*}
which according to Corollary \ref{cor2} does not have property $(N)$. Therefore the compositum of $L_1$ and $L_2$
does not have property $(N)$.

\section{Proof of Corollary \ref{corEisenstein3}}
For $i>0$ the extension $F_i/\IQ$ is generated by a root, say $\alpha_i$, of a $p_i$-Eisenstein polynomial $D_i$ in $\IZ[x]$
(see Lemma \ref{Eisensteintotram} Section \ref{sec:4}) of degree $d_i\leq d[F_0:\IQ]$. 
Therefore $F_i=\IQ(\alpha_i)$ and the compositum of $F_0,F_1,F_2,F_3,...$ is given by $F_0(\alpha_1,\alpha_2,\alpha_3,...)$.
From the hypothesis we know that $p_i\nmid \Delta_{\IQ(\alpha_j)}$
for $1\leq j<i$, in particular the primes $p_i$ are pairwise distinct and thus 
$p_i^{{1}/{d_i}}\longrightarrow \infty$. Applying Theorem \ref{corEisenstein} yields the desired result.

\section{Proof of Corollary \ref{corcorTh1}}
Write $K_i$ for the compositum of $F_0,...,F_i$. For $i>0$ we have $1\leq [K_i:K_{i-1}]\leq 3$, 
in particular $K_i/K_{i-1}$ does not admit a proper intermediate field and so (\ref{reldiscond0}) simplifies to (\ref{reldiscond}).
By assumption there is a prime $p_i$ which ramifies
in $F_i$ but not in $F_j$ for $0\leq j<i$. By virtue of (\ref{dis}) we conclude that 
\begin{alignat*}1
p_i^{[K_i:F_i]}\mid \Delta_{F_i}^{[K_i:F_i]}\mid \Delta_{K_i}.
\end{alignat*}
On the other hand 
$p_i$ does not ramify in $F_0,...,F_{i-1}$ and so does not ramify in the compositum $K_{i-1}$, that is $p_i\nmid \Delta_{K_{i-1}}$. Appealing to (\ref{dis}) again we conclude 
\begin{alignat*}1
p_i^{[K_i:F_i]}\mid N_{K_{i-1}/\IQ}(D_{K_i/K_{i-1}})
\end{alignat*}
and therefore
\begin{alignat}1\label{ineqcorcorTh1}
N_{K_{i-1}/\IQ}(D_{K_i/K_{i-1}})^{\frac{1}{[K_{i}:K_0][K_i:K_{i-1}]}}\geq p_i^{\frac{[K_i:F_i]}{[K_i:K_0][K_i:K_{i-1}]}}.
\end{alignat}
Since $[K_i:F_i]=[K_i:K_0]/[F_i:K_0]$ and $[F_i:K_0]\leq 3$ and $[K_i:K_{i-1}]\leq 3$ we see that the right hand-side of
(\ref{ineqcorcorTh1}) is at least
$p_i^{1/9}$. Now clearly 
$p_i\longrightarrow \infty$ as $i$ tends to infinity and so the statement follows from Theorem \ref{Th1}.

\section{Proof of Corollary \ref{cor2corTh1}}
Note that the primes $p_i$ are pairwise distinct. Hence there exists an integer $i_0>1$ such that $p_i>3$ for all $i\geq i_0$.
Write $\zeta_3=(-1+\sqrt{-3})/2$ and define $K_0$ as the compositum of $\IQ(\zeta_3),F_1^{(g)},...,F_{i_0}^{(g)}$ and for $i>0$ define
$K_i$ as the compositum of $K_0,F_1^{(g)},...,F_{i}^{(g)}$.
Now $\Delta_{\IQ(\zeta_3)}=-3$ and using our assumption we conclude 
that $p_i\nmid \Delta_{K_{i-1}}$ for all $i\geq i_0$.
We will show that for $i\geq i_0$ the prime $p_i$ ramifies in $M$ for each intermediate field $K_{i-1}\subsetneq M\subseteq K_i$.
By similar arguments as in the proof of Corollary \ref{corcorTh1} we derive
\begin{alignat*}1
N_{K_{i-1}/\IQ}(D_{M/K_{i-1}})^{\frac{1}{[M:K_0][M:K_{i-1}]}}\geq p_i^{\frac{1}{18}}.
\end{alignat*}
Applying Theorem \ref{Th1} proves the statement.
So let us now prove that $p_i$ ramifies in $M$.
If (a) holds we have $M=K_i$ and since 
$p_i\mid \Delta_{F_i}\mid \Delta_{K_i}$ we are done.
Next suppose (b) holds. Since $\zeta_3$ lies in $K_0$ we have $[K_i:K_{i-1}]\leq 3$
and so $M=K_i$ as before.
Now suppose (a) does not hold. 
Then $F_i^{(g)}/\IQ$  must have Galois group isomorphic to $S_3$.
The unique quadratic subfield, let us call it $E_i$, is then given by $\IQ(\sqrt{Disc(D)})$
where $D$ is the minimum polynomial of any $\alpha$ with $F_i=\IQ(\alpha)$.
Note that $[K_i:K_{i-1}]=[K_{i-1}F_i^{(g)}:K_{i-1}]=[K_{i-1}F_i^{(g)}:K_{i-1}E_i][K_{i-1}E_i:K_{i-1}]\leq [F_i^{(g)}:E_i][E_i:\IQ]=3\cdot 2$.
Hence if $K_i/K_{i-1}$ has Galois group isomorphic to $S_3$ then each strict intermediate field of $K_i/K_{i-1}$ is either the compositum of $K_{i-1}$ and a conjugate field of $F_i/\IQ$
or the compositum of $K_{i-1}$ and $E_i$.
Since $p_i$ ramifies in all conjugate fields of $F_i/\IQ$
it remains to show that $p_i$ ramifies in $E_i$.
Suppose (c) holds. Write $Q$ for the largest square dividing $Disc(D_{\alpha})$ and set $A=Disc(D_{\alpha})/Q$. Then 
$p_i\mid A$ and $A\mid \Delta_{E_i}$.
In particular $p_i$ ramifies in $E_i$. 
Now suppose (d) holds.
By Corollary 1 of \cite{Osada} we see that in this case
$F_i^{(g)}/E_i$ is unramified at all finite primes. Since $p_i$ ramifies in $F_i^{(g)}$ it must already ramify in $E_i$.
This shows that for $i\geq i_0$ the prime $p_i$ ramifies in $M$ for each intermediate field $K_{i-1}\subsetneq M\subseteq K_i$ and
thereby completes the proof.

\section{Some applications of the Northcott property}
\label{sec:0}
Applications to algebraic dynamics were a motivation for Northcott to study heights and related finiteness properties.
Let $S$ be a set and let $f:S\longrightarrow S$ be a self map of $S$. 
When we iterate this map we obtain an orbit $O_f(\alpha)$ for each point $\alpha\in S$ 
\begin{alignat*}1
O_f(\alpha)=\{\alpha,f(\alpha),f\circ f(\alpha),f\circ f\circ f(\alpha),...\}.
\end{alignat*}
We say a point $\alpha$ in $S$ is a preperiodic point under $f$ if $O_f(\alpha)$ is a finite set.
We are interested in the case where $S=\Qbar$ and $f$ is a polynomial map.
An important problem is to decide whether there are finitely many preperiodic points (under $f$) in a given subset $T$ of $S$.
A more specific version was proposed by Dvornicich and Zannier (\cite{DvZa} Question): \it let $T$ be a subfield of $\Qbar$ and let $f\in T[x]$ be a polynomial map with $\deg f\geq 2$. Can one decide whether the set of preperiodic points in $T$ (under $f$)
is finite or infinite? \rm If $T$ is the cyclotomic closure of
a number field Dvornicich and Zannier's Theorem 2 in \cite{DvornicichZannier2007} positively answers the question by explicitly describing all polynomials $f\in T[x]$ with infinitely many preperiodic points and $\deg f\geq 2$.
If $T$ has property $(N)$ the situation is much simpler and a well-known argument due to Northcott  
(see \cite{NorthcottPeriodic} Theorem 3) answers the question in the affirmative.
\begin{theorem}[Northcott]\label{lemmaNorthcott}
Suppose $T$ is a subset of $\Qbar$ with property $(N)$ and suppose $f\in \Qbar[x]$ with $\deg f\geq 2$. Then $T$ contains only finitely many preperiodic points under $f$.
\end{theorem}
\begin{proof}
For each non-zero rational function $g\in \Qbar(x)$ there exists a positive constant $b_g<1$ such that
$H(g(\alpha))\geq b_gH(\alpha)^{\deg g}$ for all $\alpha \in \Qbar$ and $\alpha$ not a pole of $g$
(see \cite{DvZa} or \cite{Liardet} Proposition 1).
We apply this inequality with $g=f$.
Suppose $\alpha$ is preperiodic under the polynomial $f$ and $H(\alpha)>1/b_f^2>1$. Hence
$H(f(\alpha))\geq b_fH(\alpha)^{\deg f}>H(\alpha)^{\deg f-1/2}\geq H(\alpha)^{3/2}$.
Thus with $f^n$ the $n$-th iterate of $f$ we get
$H(f^n(\alpha))> H(\alpha)^{(3/2)^n}>1$ which is a contradiction for $n$ large enough.
Therefore $H(\alpha)\leq 1/b_f^2$ and since $T$ has property $(N)$ this proves the lemma.
\end{proof}
Using our results on property $(N)$ we get, presumably new, answers
on Dvornicich and Zannier's question.\\

In \cite{PolyTrans} Narkiewicz introduced the so-called property $(P)$ for fields. A field $F$ has the property $(P)$
if for every infinite subset $\Gamma \subset F$ the condition $f(\Gamma)=\Gamma$ for a polynomial $f\in F[x]$ implies $\deg f=1$.
Narkiewicz proposed several problems involving property $(P)$, e.g. the analogue of Bombieri and Zannier's question (\cite{unsolvedproblems} Problem 10 (i)): \it does $\IQ^{(d)}$ have property $(P)$? \rm Or less specifically (\cite{PolyMapp} Problem XVI): \it give a constructive description of fields with property $(P)$. \rm
Dvornicich and Zannier have noticed (\cite{DvornicichZannier2007} p. 534) that for subfields of $\Qbar$ property $(N)$ implies property $(P)$
(see also \cite{DvZa} Theorem 3.1 for a detailed proof).
Hence an affirmative answer on Bombieri and Zannier's question would 
also positively answer Narkiewicz's first problem, and the explicit examples of fields with property $(N)$ shed some light on 
Narkiewicz's second problem. 
But property $(P)$ does not imply property $(N)$ as shown in \cite{DvornicichZannier2007} Theorem 3. However,
Dvornicich and Zannier also remarked 
(\cite{DvornicichZannier2007} p. 533 and \cite{DvZa} Proposition 3.1) that the property $(P)$ already implies the finiteness of the set of preperiodic points under a polynomial map of degree at least $2$.
Kubota and Liardet \cite{KubotaLiardet} proved the existence and Dvornicich and Zannier (\cite{DvornicichZannier2007} Theorem 3) gave 
explicit examples of fields with property $(P)$ that cannot be generated over $\IQ$ by algebraic numbers of bounded degree.
These examples refuted a conjecture of Narkiewicz (\cite{PolyMapp} p.85). Corollary \ref{cor2} provides further examples of such fields 
but, opposed to Dvornicich and Zannier's example, they also have property $(N)$.

\begin{acknowledgements}
The author would like to thank Professor Umberto Zannier for sending him a preprint of \cite{DvZa} and the referee for the example just after Theorem \ref{thmcompfields} and other valuable comments which improved the exposition of this article.
\end{acknowledgements}

\bibliographystyle{amsplain}      
\bibliography{literature}


%
%

\end{document}